\UseRawInputEncoding
\documentclass{amsart}
\usepackage[all,cmtip]{xy}
\usepackage{amsthm}
\usepackage{microtype}
\usepackage{amssymb}
\usepackage{xypic}
\usepackage{graphicx}
\usepackage{amscd}
\usepackage{setspace}
\usepackage{xcolor}
\usepackage{mathtools}
\usepackage[colorlinks=false, linkcolor=red]{hyperref}

\DeclareMathOperator{\Spec}{Spec}
\DeclareMathOperator{\Spf}{Spf}
\DeclareMathOperator{\Spa}{Spa}
\DeclareMathOperator{\Sp}{Sp}
\DeclareMathOperator{\Hom}{Hom}
\DeclareMathOperator{\Pic}{Pic}
\DeclareMathOperator{\Frac}{Frac}
\DeclareMathOperator{\Fil}{Fil}
\DeclareMathOperator{\Gr}{Gr}
\DeclareMathOperator{\adic}{adic}

\newtheorem{theorem}{Theorem}[section]
\newtheorem{thm}[theorem]{Theorem}

\newtheorem{lemm}[theorem]{Lemma}
\newtheorem{prop}[theorem]{Proposition}

\theoremstyle{definition}
\newtheorem{defi}[theorem]{Definition}

\newtheorem{example}[theorem]{Example}

\theoremstyle{remark}
\newtheorem{remark}[theorem]{Remark}
\newtheorem{question}[theorem]{Question}

\numberwithin{equation}{section}

\title[]{On rigid varieties with projective reduction}

\author{Shizhang Li}
\address{Department of Mathematics, Columbia University}
\email{shanbei@math.columbia.edu}



\begin{document}
\maketitle

\begin{abstract}
  In this paper, we study smooth proper rigid varieties which admit formal
  models whose special fibers are projective. The main theorem asserts that the
  identity components of the associated rigid Picard varieties will
  automatically be proper. Consequently, we prove that p-adic Hopf varieties
  will never have a projective reduction. The proof of our main theorem uses the
  theory of moduli of semistable coherent sheaves.
\end{abstract}

\tableofcontents

% * Introduction
\section{Introduction}

In the famous series~\cite{BL1},~\cite{BL2},~\cite{BLR3} and~\cite{BLR4}, Bosch,
L\"{u}tkebohmert and Raynaud laid down the foundations relating formal and rigid
geometry. The type of questions they treat in the series are mostly concerned
with going from the rigid side to formal side. In this paper we will consider
the opposite type of question, namely we will investigate to what extent
properties on the formal side inform us about rigid geometry. More precisely, we
will see what geometric consequences one can deduce under the assumption that
the rigid space has a projective reduction.

Let \(K\) be a non-archimedean field with residue field \(k\). Let \(X\) over
\(K\) be a connected smooth proper rigid space with a \(K\)-rational point \(x:
\Sp(K) \to X\).

In this paper we prove the following:

\begin{thm}
\label{proper picard}
Suppose that \(X\) has a formal model \(\mathcal{X}\) whose special fiber
\(\mathcal{X}_0\) is projective over \(\Spec(k)\), assume furthermore that
Picard functor is represented by a quasi-separated rigid space. Then its
identity component \(\Pic^0_X\) is proper.
\end{thm}

Note that if \(K\) is a \(p\)-adic field, then the properness of \(\Pic^0_X\)
implies Hodge symmetry of degree \(1\): \(h^{1,0}(X) = h^{0,1}(X)\)
(c.f.~\cite[Theorem 1.2]{HL17}).

Let us make a historical remark on the representability of Picard functor in
rigid geometry.

\begin{remark}
  \leavevmode
  \begin{enumerate}
  \item When \(K\) is discretely valued, Hartl and L\"{u}tkebohmert proved the
    representability of the Picard functor on the category of smooth rigid
    spaces over \(K\) under an additional assumption that \(X\) has a strict
    semistable formal model (c.f.~\cite{HL}). They prove a structure theorem for
    the Picard space. In particular, the it is quasi-separated.
  \item In Warner's thesis it is proved that assuming \(K\) has characteristic
    \(0\), the Picard functor defined on a suitable category of adic spaces over
    \(K\) is represented by a separated rigid space over \(\Spa(K,\mathcal{O})\)
    (c.f.~\cite[Theorem 1.0.2]{Warner}).
  \item In general, we expect the Picard functor of a proper smooth rigid space
    over \(K\) to be represented by a separated rigid space.
  \end{enumerate}
\end{remark}

If \(X\) has a projective reduction then one can naturally define an open and
closed sub-functor \(\underline{\Pic}^{P}_{X/K}\) of \(\underline{\Pic}_{X/K}\),
see Definition~\ref{pictau}.

\begin{thm}[Main Theorem]
\label{Main Theorem}
Suppose that \(X\) has a formal model \(\mathcal{X}\) whose special fiber
\(\mathcal{X}_0\) is projective over \(\Spec(k)\). Then
\(\underline{\Pic}^{P}_{X/K}\) is a proper functor.
\end{thm}

The notion of a functor being proper is defined in Section~\ref{Preliminaries on
  Picard functor}. By Theorem~\ref{partially proper}, it suffices to prove
\(\underline{\Pic}^{P}_{X/K}\) is bounded. We use moduli of semistable sheaves
to show this.

% If Picard functor is represented by \(\Pic_X\) and denote its identity component
% as \(\Pic^0_X\), then by definition it is obvious that every connected component
% of \(\Pic^{\tau}_X\) is isomorphic to \(\Pic^0_X\).

Let us summarize how this paper is organized. In Section~\ref{Specialization of
  K group} we make an observation that there is a well defined specialization
from the \(K\)-group of coherent sheaves on \(X\) to that of \(\mathcal{X}_0\).
Therefore a chosen ample invertible sheaf on \(\mathcal{X}_0\) will enable us to
associate Hilbert polynomials. As a byproduct, we find that non-archimedean Hopf
surfaces do not have any formal model with projective reduction, see
Proposition~\ref{Hopf reduction}.

From Section~\ref{Semistable Reduction} on, we fix a chosen ample invertible
sheaf on \(\mathcal{X}_0\). Then we define (semi)-stability of coherent sheaves
on \(X\) and generalize a result of Langton, namely we prove that a semistable
coherent sheaf \(F\) on \(X\) always has a formal model \(\mathcal{F}\) such
that its reduction \(\mathcal{F}_0\) is semistable, which justifies the title of
this section. Consequently, any line bundle on \(X\) has a formal model whose
reduction is semistable.

In the last section we construct an auxiliary quasi-compact space \(W\) which
``surjects'' onto \(\underline{\Pic}^{P}_X\); this completes the proof of our
main theorem.

% * Notations
\section{Notations}

Let \(K\) be a non-archimedean field with value group \(\Gamma \subset
\mathbb{R}\), ring of integers \(\mathcal{O}\), maximal ideal \(\mathfrak{m}\)
and residue field \(k\). Throughout this paper, \(\mathcal{X}\) will be a proper
admissible formal model of a smooth proper connected rigid space
\(X=\mathcal{X}^{\mathrm{rig}}\) over \(\mathcal{O}\). For simplicity we will
assume \(X\) has a \(K\)-rational point \(x: \Sp(K) \to X\). We use almost the
same notation as in~\cite{BL1} except that we use \({(\cdot)}^{\mathrm{rig}}\)
to denote the generic fiber of an admissible formal scheme. We use roman letters
to denote rigid objects and curly letters to denote formal objects. We also
denote the level \(\pi\) (resp.\ \(\gamma\)), namely modulo \(\pi\) (resp.\
modulo \(\pi_{\gamma}\) for some \(|\pi_{\gamma}|=\gamma\)), of a formal objects
by subscript \(\pi\) (resp.\ \(\gamma\)).

% * Preliminaries on Picard functor
\section{Preliminaries on Picard functor}
\label{Preliminaries on Picard functor}

In this section, we will have a general discussion of (Picard) functor in the
context of rigid spaces and adic spaces.

Following~\cite{Warner}, we consider functors defined on a suitable subcategory
of adic spaces over \(K\). Let us denote \((\mathbf{V}/K)\) the full subcategory
of adic spaces over \(\Spa(K,\mathcal{O})\), its objects are adic spaces locally
of finite type over \(\Spa(K',R') \to \Spa(K,\mathcal{O})\) where \(K'/K\) is a
non-archimedean field extension, \(R'\) gives a (not necessarily rank 1)
valuation on \(K'\) with \(R' \cap K = \mathcal{O}\). In the following we will
use underlined notation to denote a functor, and the representing space (if
exists) will be denoted without underline. Note that this is opposite to the
convention of Grothendieck but agrees with that in loc.~cit.

In this paper, we consider the Picard functor
\[
  \underline{\Pic}_{X/K} \colon (\mathbf{V}/K)^{op} \to \mathcal{S}ets
\]
defined by
\[
  \underline{\Pic}_{X/K}(V)=\left\{\parbox{16em}{$\mathcal{L}$ a line bundle on
      $X^\mathrm{ad} \times_{\Spa(K,\mathcal{O})} V$, $\lambda: \mathcal{O}_V
      \xrightarrow{\sim} {(x,id)}^*\mathcal{L}$ an
      isomorphism}\right\}/\mathrm{Isom}.
\]
Note that this is a special case of the Picard functor defined in~\cite{Warner}.
Let us make the following definition.

\begin{defi}
\label{definitions of functor}
  Let \(F \colon (\mathbf{V}/K)^{op} \to \mathcal{S}ets\) be a functor.
  \begin{enumerate}
  \item \(F\) is said to be \emph{partially proper} if any element in
    \(F(\Spa(\Frac(R),R))\) can be extended to a unique element in
    \(F(\Spa(\Frac(R),R^+))\). Here \(R\) is a rank 1 valuation ring over
    \(\mathcal{O}\) and \(R^+\) is a valuation subring of \(R\) with the same
    fraction field.
  \item \(F\) is said to be \emph{bounded} if there exists a quasi-compact rigid
    space \(W\) over \(K\) and an element \(\xi \in F(W)\) such that for any
    \(\eta \in F(\bar{K})\), there exists \(P \in W(\bar{K})\) with \(\eta \cong
    P^*(\xi)\).
    \item \(F\) is said to be \emph{proper} if it is partially proper
  and bounded.
  \end{enumerate}
\end{defi}

In the definition of boundedness of functor, it is equivalent if we ask the
space \(W\) to be an affinoid space. This definition is justified by
Proposition~\ref{justify} below.

\begin{prop}
\label{justify}
Let \(F \colon (\mathbf{V}/K)^{op} \to \mathcal{S}ets\) be a (partially) proper
functor. If \(F\) is represented by a quasi-separated rigid space \(Y\) over
\(\Sp(K)\). Then \(Y\) is (partially) proper.
\end{prop}

Huber has defined the notion of partial properness (c.f.~\cite[Definition
1.3.3]{Huber}) for analytic adic spaces. He showed a valuative criterion of
partial properness for maps between analytic adic spaces (c.f.~\cite[Corollary
1.3.9]{Huber}). Combined with the work in~\cite{formal-rigid}, Huber was able to
show that a map of rigid spaces over discretely valued field is (partially)
proper if and only if the map of corresponding analytic adic spaces is
(partially) proper (c.f.~\cite[Remark 1.3.19]{Huber}). Temkin generalized the
result of L\"{u}tkebohmert, c.f.~\cite[Theorem 4.1]{Temkin}. In particular, we
now know that a quasi-separated map of rigid spaces over an arbitrary
non-archimedean field is partially proper if and only if the map of
corresponding analytic adic spaces satisfies the valuative criterion
in~\cite[Lemma 1.3.10]{Huber}.

\begin{proof}
  By definition of \(F\) being partially proper, we see that \(Y\) satisfies the
  valuative. Since \(Y\) is quasi-separated, it is partially proper
  by~\cite[Corollary 1.3.9]{Huber}.

  If \(F\) is proper, then by the above we know that \(Y\) is partially proper.
  It suffices to show that \(F\) being bounded implies that \(Y\) is
  quasi-compact. This follows from the Lemma below.
\end{proof}

\begin{lemm}
  Let \(f: W \to Y\) be a morphism of rigid spaces. Suppose \(W\) is
  quasi-compact, \(Y\) is quasi-separated and \(f\) is surjective on classical
  Tate points. Then \(Y\) is also quasi-compact.
\end{lemm}

\begin{proof}
  Let \(\{U_i\} _{i \in I}\) be an affinoid admissible covering of \(Y\).
  Because \(W\) is quasi-compact and \(f\) is surjective we see that finitely
  many of the \(U_i\)'s cover \(Y\) set-theoretically. 
  It suffices to prove that for any affinoid admissible open \(V\) in \(Y\) the
intersection \(\{V ∩ U_i\}\) admits a refinement by a finite admissible covering. 
Now \(V\) is affinoid, so \(\{V ∩ U_i\}\) are quasi-compact, since Y is quasi-separated.
Therefore we see that \(Y\) is quasi-compact.
\end{proof}

It is worth noticing the following (counter)-example if we drop the assumption
of \(Y\) being quasi-separated.

\begin{example}
  Let \(K = \mathbb{C}_p\) and let \(W\) be the closed disc of radius \(1\).
  Choose an irrational number \(r\) with \(0 < r < 1\). Let \(Y\) be the gluing
  of infinitely many closed discs of radius \(1\), where we glue all the radius
  greater than \(r\) parts together and do the same to all the radius less than
  \(r\) parts. There is an obvious map from \(W\) to \(Y\) that is a surjection
  on classical Tate points. It is easy to see that \(Y\) is not quasi-compact.
\end{example}

It is a consequence of~\cite[Theorem 2.3.3]{K-L} that \(\underline{\Pic}_{X/K}\)
is always partially proper, see Theorem~\ref{partially proper} below.

\begin{thm}
\label{partially proper}
\(\underline{\Pic_X}_{X/K}\) is partially proper.
\end{thm}

\begin{proof}
  Let \(R\) and \(R^+\) be as in Definition~\ref{definitions of functor}.
  
  By~\cite[Theorem 2.3.3]{K-L} we see that the restriction functor from the
  category of coherent sheaves on \(X^{\adic} \times_{\Spa(K,\mathcal{O})}
  \Spa(\Frac(R),R^+)\) to the category of coherent sheaves on \(X^{\adic}
  \times_{\Spa(K,\mathcal{O})} \Spa(\Frac(R),R)\) is an equivalence and
  preserves finite locally free coherent sheaves. Indeed, Theorem 2.3.3 of 
  loc.~cit.~establishes the above equivalence of categories 
  except for the last statement about preserving finite locally free coherent sheaves.
  However, note that in affinoid setting, with the notation in loc.~cit., 
  a coherent sheaf \(\mathcal{F}\) is finite locally free is equivalent to
   the associated \(A\)-module \(M\) being finite projective. Since this condition only depends on \(A\) but not on \(A+\),
   we see that being finite locally free does not depend on the \(\mathcal{O}^+\) structure. 
   Therefore the equivalence of category above preserves finite locally free ones.
  
  Let \(\mathcal{L}_R\) on
  \(X^{\adic} \times_{\Spa(K,\mathcal{O})} \Spa(\Frac(R),R)\) be a rigidified
  line bundle. By the equivalence of categories above we see that \(\mathcal{L}_R\)
  extends uniquely as a rigidified line bundle \(\mathcal{L}_{R^+}\) on
  \(X^{\adic} \times_{\Spa(K,\mathcal{O})} \Spa(\Frac(R),R^+)\).

  The argument above shows that the Picard functor satisfies the valuative
  criterion, namely any map \[f: \Spa(\Frac(R),R) \to \underline{\Pic}_{X/K}\]
  always extends to a map \[f: \Spa(\Frac(R),R^+) \to \underline{\Pic}_{X/K}\]
  which is what we need to show.
\end{proof}

% * Specialization
\section{Specialization of \(K\) group}
\label{Specialization of K group}

Suppose we have a class in \(K_0(X)\) represented by a coherent sheaf \(F\).
Then we can find some formal model \(\mathcal{F}\) of \(F\), by which we mean an
\(\mathcal{O}\)-torsion free finitely presented \(\mathcal{O_X}\)-module with
generic fibre isomorphic to \(F\). After reduction we get a coherent sheaf
\(\mathcal{F}_0\) on \(\mathcal{X}_0\). Different formal models of \(F\) will
differ by an \(\mathcal{O}\)-torsion finitely presented sheaf on
\(\mathcal{X}\). Let us prove a lemma on the \(\mathcal{O}\)-module structure of
such sheaves which we believe is interesting on its own.

\begin{lemm}
\label{torsion structure}
  Let \(\mathcal{A}\) be a topologically finitely presented
  \(\mathcal{O}\)-algebra and let \(M\) be an \(\mathcal{O}\)-torsion finitely
  presented \(\mathcal{A}\)-module. Then as an \(\mathcal{O}\)-module, we have
  the following decomposition:
  \[M=\bigoplus_{i=1}^k{(\mathcal{O}/\pi_i)}^{\oplus \lambda_i}
  \]
  for some (finite or countable) cardinals \(\lambda_1,\ldots,\lambda_k\).
\end{lemm}

\begin{proof}
  Without loss of generality, we may assume
  \(\mathcal{A}=(\mathcal{O}/\pi)[x_1,\ldots,x_n]/I\) with \(I\) finitely
  generated. We will do induction on the dimension of the support of \(M\),
  which we denote as \(d\). Now without loss of generality we may assume the
  dimension of \(|\Spec(\mathcal{A})|=|\Spec(\mathcal{A} \otimes_{\mathcal{O}}
  k)|\) is \(d\). By Noether normalization, we may find a finite morphism
  \(k[x_1,\ldots,x_d] \to \mathcal{A} \otimes_{\mathcal{O}} k\). Lifting the
  images of \(x_i\)'s gives us a morphism
  \(\mathcal{B}=\mathcal{O}[x_1,\ldots,x_d] \to \mathcal{A}\) which is
  universally closed and of finite presentation as an algebra, hence of finite
  presentation as a module. Now regarding \(M\) as a finitely presented
  \(\mathcal{B}\)-module, we reduce to the situation where
  \(\mathcal{A}=\mathcal{B}\). Localizing at the generic point of the special
  fiber \(\mathfrak{p}=\mathfrak{m}[x_1,\ldots,x_d]\), we see that
  \(M_{\mathfrak{p}}\) is a finitely presented
  \(\mathcal{B}_{\mathfrak{p}}\)-module. Because \(\mathcal{B}_{\mathfrak{p}}\)
  is a valuation ring which is unramified over \(\mathcal{O}\), i.e., they have
  the same value group. We see that
  \(M_{\mathfrak{p}}=\bigoplus_{i=1}^k(\mathcal{B}_{\mathfrak{p}}/\pi_i\mathcal{B}_{\mathfrak{p}})\).
  By clearing denominator, this gives rise to a morphism
  \(\bigoplus_{i=1}^k\mathcal{B} \rightarrow M\), and hence an injection
  \(\phi: \bigoplus_{i=1}^k(\mathcal{B}/\pi_i) \hookrightarrow M\) after
  multiplying some element \(f \notin \mathfrak{p}\).

  Next, we claim that \(\phi\) is universally injective as an
  \(\mathcal{O}\)-module and the quotient is a finitely presented
  \(\mathcal{B}\)-module whose support has dimension less than \(d\). The claim
  on the dimension of the support is easy and follows from the fact that this
  map is an isomorphism on the generic point of the special fiber
  \(\mathfrak{p}\). To see why this map is universally injective as an
  \(\mathcal{O}\)-module, let us consider the following commutative diagram
  \[
    \xymatrix{
      \bigoplus_{i=1}^k(\mathcal{B}/\pi_i) \ar[d] \ar@{^{(}->}[r] & M \ar[d] \\
      \bigoplus_{i=1}^k(\mathcal{B}_{\mathfrak{p}}/\pi_i\mathcal{B}_{\mathfrak{p}})
      \ar[r]^-{\sim} & M_{\mathfrak{p}}}.
  \]
  Now by~\cite[Lemma \href{http://stacks.math.columbia.edu/tag/058H}{Tag
    05CI}]{stacks-project} and~\cite[Lemma
  \href{http://stacks.math.columbia.edu/tag/058H}{Tag 05CJ}]{stacks-project}, it
  suffices to show that \(\mathcal{B}/\pi_i \to
  \mathcal{B}_{\mathfrak{p}}/\pi_i\mathcal{B}_{\mathfrak{p}}\) is universally
  injective as an \(\mathcal{O}\)-module. This in turn would follow from the
  fact that \(\mathcal{B} \to \mathcal{B}_{\mathfrak{p}}\) is universally
  injective as an \(\mathcal{O}\)-module, by~\cite[Theorem
  \href{http://stacks.math.columbia.edu/tag/058H}{Tag 058K}]{stacks-project} it
  suffices to show \(\mathcal{B}/\pi_i \to
  \mathcal{B}_{\mathfrak{p}}/\pi_i\mathcal{B}_{\mathfrak{p}}\) is injective as
  an \(\mathcal{O}\)-module which one verifies directly.

  Consider the short exact sequence \(0 \to \bigoplus_{i=1}^k(\mathcal{B}/\pi_i)
  \to M \to Q \to 0\). It is easy to see that
  \(\bigoplus_{i=1}^k(\mathcal{B}/\pi_i)\) has the form we want. We see that
  this sequence is universally exact with respect to \(\mathcal{O}\)-module
  structure and \(Q\) is a finitely presented \(\mathcal{B}/(f)\)-module for
  some \(f \notin \mathfrak{m}[x_1,\ldots,x_d]\) and therefore is of the form
  \(Q=\bigoplus_{i=1}^{k'}{(\mathcal{O}/\pi_i')}^{\oplus \lambda_i'}\) by the
  induction hypothesis. Now by~\cite[Theorem
  \href{http://stacks.math.columbia.edu/tag/058H}{Tag 058K}]{stacks-project}, we
  see that the short exact sequence above splits. Therefore we see that \(M\),
  as an \(\mathcal{O}\)-module, has the form we want.
\end{proof}

Now we can state and prove the key observation of this paper.

\begin{thm}
\label{well-defined}
The association \([F] \mapsto [\mathcal{F}_0]\) gives rise
to a well-defined map from \(K_0(X)\) to \(K_0(\mathcal{X}_0)\).
\end{thm}
 
\begin{proof}
  The first thing we have to verify is that for two different formal models of
  \(F\), say \(\mathcal{F}^1\) and \(\mathcal{F}^2\), we will have
  \([\mathcal{F}^1_0]=[\mathcal{F}^2_0]\) as classes in \(K_0(\mathcal{X}_0)\).
  Viewing \(\mathcal{F}^1\) and \(\mathcal{F}^2\) as two lattices inside \(F\),
  then after multiplying \(\mathcal{F}^1\) by an element in \(\mathfrak{m}\)
  with sufficiently large valuation, we may assume it is a subsheaf of
  \(\mathcal{F}^2\).

  In the situation above, we have a short exact sequence
  \[
  0 \to \mathcal{F}^1 \to \mathcal{F}^2 \to Q \to 0
 \]

 Applying \(\otimes_{\mathcal{O_X}}\mathcal{O_X}/\mathfrak{m}\mathcal{O_X}\) we
 get
  \[
  0 \to
  \mathrm{Tor}_{\mathcal{O_X}}^1(\mathcal{O_X}/\mathfrak{m}\mathcal{O_X},Q) \to
  \mathcal{F}^1_0 \to \mathcal{F}^2_0 \to Q_0 \to 0.
  \]
  Observe that
  \(\mathrm{Tor}_{\mathcal{O_X}}^1(\mathcal{O_X}/\mathfrak{m}\mathcal{O_X},Q)=\ker(\mathfrak{m}
  \otimes Q \to Q)\), which we shall denote as \(Q'_0\).

  Now we only have to show \([Q'_0]=[Q_0]\). We will define canonical
  filtrations on both coherent sheaves and use Lemma~\ref{torsion structure} to
  see their successive quotients are naturally isomorphic. For every \(\gamma
  \in \Gamma\), we define
\[\Fil^{\gamma}(Q_0)=\mathrm{Im}(\ker(\cdot\pi_{\gamma}: Q \to Q) \to Q_0)\] and
\[\Fil^{\gamma}(Q'_0)=(\pi_{\gamma} \otimes Q) \cap \ker(\mathfrak{m}
  \otimes Q \to Q)
  \]
  where \(\pi_{\gamma}\) is any element in \(\mathcal{O}\) with valuation
  \(\gamma\). Since \(Q\) is an \(\mathcal{O}\)-torsion finitely presented
  \(\mathcal{O_X}\)-module, we see that as \(\gamma\) goes from \(0\) to some
  \(|\pi|\) the first (resp.\ second) filtration is an increasing (resp.\
  decreasing) filtrations of coherent subsheaves and both filtrations are
  exhaustive and separated. We claim that both filtrations only change at
  finitely many \(\gamma_i\) and there are natural isomorphisms between
  corresponding successive quotients. Because both claims are local properties,
  after choosing a finite covering of \(\mathcal{X}\) by affine formal schemes
  we may reduce to the situation where \(\mathcal{X}=\Spf(\mathcal{A})\). Now we
  have \(Q=\widetilde{M}\) for some \(\mathcal{O}\)-torsion finitely presented
  \(\mathcal{A}\)-module \(M\). By Lemma~\ref{torsion structure} we see that as
  an \(\mathcal{O}\)-module,
  \[M=\bigoplus_{i=1}^k{(\mathcal{O}/\pi_i)}^{\oplus \lambda_i}.\] We may assume
  \(|\pi_1| < \cdots < |\pi_k|\). Now by direct computation, we see that both
  filtrations only change when \(\gamma=|\pi_i|\) for \(1 \leq i \leq k\). On
  the successive quotient we may define
  \begin{align*}f_i: &\Gr^i(Q_0) =\frac{\mathrm{Im}(\ker(\cdot\pi_i: Q \to Q) \to
      Q_0)}{\mathrm{Im}(\ker(\cdot\pi_{i+1}: Q \to Q) \to Q_0)} \\
      &\to \Gr^i(Q'_0)=\frac{(\pi_i \otimes Q) \cap \ker(\mathfrak{m} \otimes Q \to
      Q)}{(\pi_{i-1} \otimes Q) \cap \ker(\mathfrak{m} \otimes Q \to Q)}
  \end{align*}
  by \(f_i(x)=\overline{\pi_i \tilde{x}}\) where \(\tilde{x}\) is any lift of
    \(x\) in \(\mathrm{Im}(\ker(\cdot\pi_i: Q \to Q) \to Q_0)\). Again one
    checks directly that \(f_i\) is a well-defined isomorphism and commutes with
    localization. This proves our claim that \([Q'_0]=[Q_0]\), which implies
    that the class \([\mathcal{F}_0]\) is well defined.

    Secondly we note that every short exact sequence of coherent sheaves on
    \(X\) will have a formal model, i.e., a short exact sequence of
    \(\mathcal{O}\)-flat finitely presented sheaves on \(\mathcal{X}\). Indeed,
    consider \( 0 \to F \to G \to H \to 0 \) a short exact sequence of coherent
    sheaves on \(X\). We choose a formal model of \(G\), say \(\mathcal{G}\).
    Then \(\mathcal{F} := F \cap \mathcal{G} \) is a saturated finitely
    presented (c.f.~\cite[Proposition 1.1 and Lemma 1.2]{BL1}) sheaf whose
    generic fiber is \(F\). Hence \( \mathcal{H} := \mathcal{G}/\mathcal{F}\) is
    a \(\mathcal{O}\)-flat finitely presented sheaf on \(\mathcal{X}\) whose
    generic fiber is \(H\). Now because \(\mathcal{H}\) is \(\mathcal{O}\)-flat,
    after tensoring with \(k\) we get a short exact sequence \( 0 \to
    \mathcal{F}_0 \to \mathcal{G}_0 \to \mathcal{H}_0 \to 0\). So we see this
    definition respects the relation coming from short exact sequences. This
    proves our theorem.
\end{proof}

\begin{remark}
\label{S-equivalence}
From the proof of Theorem~\ref{well-defined} we see that for any two choices of
formal models \(\mathcal{F}^i\) of \(F\), we can define a decreasing filtration
\(\Fil^i(\mathcal{F}^1_0)\) and an increasing filtration
\(\Fil^i(\mathcal{F}^2_0)\) along with maps \(\phi^i: \Fil^i(\mathcal{F}^1_0)
\to (\mathcal{F}^2_0)/(\Fil^i(\mathcal{F}^2_0))\) such that these maps give rise
to isomorphisms between \(\Gr^i(\mathcal{F}^1_0)\) and
\(\Gr^i(\mathcal{F}^2_0)\). This will be used later, see
Theorem~\ref{S-equivalence theorem}.
\end{remark}

From now on we will assume that \(\mathcal{X}_0\) is projective, with a fixed
ample invertible sheaf \(H\). Consider \(\Spa(K',R') \to \Spa(K,\mathcal{O})\)
as in the definition of the category \(\mathbf{V}/K\). We may form a formal
model of \(X_{K'}\) by taking \(\mathcal{X} \times_{\Spf{\mathcal{O}}}
\Spf{\mathcal{O}_{K'}}\). The special fibre of this formal model will be
\(\mathcal{X}_0 \times_{k} k'\) and carry the pullback of \(H\) which is still
ample, where \(k'\) is the residue field of \(\mathcal{O}_{K'}\). Then we can
define the Hilbert polynomial of a coherent sheaf on \(X\) (or more generally
\(X^{\mathrm{ad}} \times_{\Spa(K,\mathcal{O})} \Spa(K',R')\)) in the following
way:

\begin{defi}
  For any coherent sheaf \(F\) on \(X\), we define the \emph{Hilbert polynomial}
  of \(F\) to be
  \[
    P_H(F)(n)=\chi_{\mathcal{X}_0}(Sp(F) \otimes H^{\otimes n})
  \]
  where \(Sp\) is the specialization map of K groups from previous theorem.

  More generally for any coherent sheaf \(F\) on \(X^\mathrm{ad}
  \times_{\Spa(K,\mathcal{O})} \Spa(K',R')\), we define its Hilbert polynomial
  to be that of its restriction to \(X^\mathrm{ad} \times_{\Spa(K,\mathcal{O})}
  \Spa(K', \mathcal{O}_{K'})\) with respect to the formal model \(\mathcal{X}_0
  \times_{k} k'\) and the pullback of \(H\).
\end{defi}

\begin{prop}
\label{locally constant}
  Suppose \(F \in Coh(X \times S)\) is a flat family of coherent sheaves on
  \(X\) parametrized by a quasi-compact quasi-separated connected rigid space
  \(S\), and suppose \(\mathcal{X}_0\) is projective with an ample invertible
  sheaf \(H\). Then any two fibres of \(F\) will have the same Hilbert
  polynomial (i.e., the Hilbert polynomial is locally constant in flat
  families).

  More generally if \(F \in Coh(X^{\mathrm{ad}} \times_{\Spa(K,\mathcal{O})}
  V)\) is a flat family of coherent sheaves with \((V \to \Spa(K',R')) \in
  \mathbf{V}/K\). Then the Hilbert polynomial function gives a partition of
  \(V\) into union of connected components.
\end{prop}

\begin{proof}
  Let \(\mathcal{S}\) be a formal model of \(S\). Then \(\mathcal{X} \times
  \mathcal{S}\) is a formal model of \(X \times S\), and \(F\) extends to a
  coherent sheaf \(\mathcal{F}\) on \(\mathcal{X} \times \mathcal{S}\).

  Applying~\cite[Theorem 4.1]{BL2}, we see that after a possible admissible
  formal blowing-up of \(\mathcal{S}\) we may assume that \(\mathcal{F}\) is
  flat over \(\mathcal{S}\). Hence \(\mathcal{F}_0\) will be flat over
  \(\mathcal{S}_0\). Then it follows from cohomology and base change
  (c.f.~\cite[Corollary 1 in Section 5]{Mumford}) that the reduction of two
  special fibers of \(F\) have the same Hilbert polynomial.

  We see that the Hilbert polynomial is locally constant in a flat family, so
  the statement above would still hold even if we do not assume
  quasi-compactness or quasi-separatedness of \(S\).

  To prove the more general statement it suffices to notice that for a locally
  finite type adic space \(V\) over \(\Spa(K',R')\), its connected components
  are in bijection with the connected components of \(V \times_{\Spa(K',R')}
  \Spa(K',\mathcal{O}_{K'})\). Now by locally constancy in the case of rigid
  space proved above, we see that the Hilbert polynomial gives rise to a
  partition of \(V \times_{\Spa(K',R')} \Spa(K',\mathcal{O}_{K'})\) into union
  of connected components.
\end{proof}

The above discussion already yields the following interesting consequence for
the geometry of non-archimedean Hopf surfaces, which were defined and studied by
Voskuil in~\cite{Hopf}.

\begin{prop}
\label{Hopf reduction}
  Non-archimedean Hopf surfaces over a non-archimedean field have no projective
  reduction.
\end{prop}

\begin{proof}
  In the paper mentioned above, Voskuil proved that there is a flat family of
  line bundles on a Hopf surface parametrized by \(\mathbb{G}_m\) with identity
  in \(\mathcal{G}_m\) corresponds to \(\mathcal{O}_X\). Moreover, he proved
  that there are infinitely many nontrivial (i.e., not isomorphic to
  \(\mathcal{O}_X\)) line bundles in that family which possess sections.

  Now suppose that a Hopf surface \(X\) has a projective reduction. Then we can
  define Hilbert polynomials as above. Because \(\mathbb{G}_m\) is connected,
  the line bundles it parametrizes will have the same Hilbert polynomial by
  Proposition~\ref{locally constant}. Let \(\mathcal{L}\) be a nontrivial line
  bundle possessing a nontrivial section, so that we get an injection \(0 \to
  \mathcal{O}_X \to \mathcal{L} \to Q \to 0\). We observe immediately that the
  cokernel \(Q\) has zero Hilbert polynomial, which means \(Q\) is a zero sheaf.
  This is a contradiction as \(\mathcal{L}\) is assumed to be a nontrivial line
  bundle.
\end{proof}

Using Hilbert polynomial we may define a natural partition on any moduli functor
of coherent sheaves, as illustrated below.

\begin{defi}
\label{pictau}
Assume \(X\) has a projective reduction and fix a projective reduction
\(\mathcal{X}_0\) along with an ample line bundle on it. Let
\(\underline{\Pic}^{P}_{X/K}\) be the open and closed sub-functor of
\(\underline{\Pic}_{X/K}\) parametrizing line bundles having Hilbert polynomial
the same as that of \(\mathcal{O}_X\).
\end{defi}

\begin{remark}
  \leavevmode
  \begin{enumerate}
  \item This is an open and closed sub-functor because of
    Proposition~\ref{locally constant}.
  \item It is not clear to the author if the definition above depends on choices
    of the projective reduction and the ample line bundle on it. The original
    purpose of this paper was on the behavior of \(\Pic^0_X\), regarding this
    direction we do not have to care about this issue.
  \end{enumerate}
\end{remark}

% * Semistable Reduction
\section{Semistable Reduction of Coherent Sheaves}
\label{Semistable Reduction}

In this section we prove a semistable reduction type theorem for semistable
coherent sheaves, following a similar method as in~\cite{Langton}.
Following~\cite{Huybrechts-Lehn}, we make the definitions below. From now on we
will assume all the formal schemes appearing below have projective special
fibers.

\begin{defi}
  Let \(\mathcal{F}\) be an \(\mathcal{O}\)-flat finitely presented sheaf on
  \(\mathcal{X}\).
  \begin{enumerate}
  \item The \emph{Hilbert polynomial} of \(\mathcal{F}\) is the Hilbert
    polynomial of \(\mathcal{F}_0\).
  \item The \emph{dimension} of \(\mathcal{F}\) is the dimension of its support,
    denoted as \(\dim \mathcal{F}\).
  \item The \emph{rank} of \(\mathcal{F}\) is the leading coefficient of its
    Hilbert polynomial divided by that of \(\mathcal{O}_X\), denoted as
    \(\mathrm{rk}(\mathcal{F})\).
  \item \(\mathcal{F}\) is said to be \emph{pure} if any non-trivial coherent subsheaf of
    \(\mathcal{F}\) with a \(\mathcal{O}\)-flat quotient has dimension \(n=\dim(\mathcal{F})\).
  \item The \emph{dimension} of a coherent sheaf \(F\) on \(X\) is the dimension
    of its support. It is called \emph{pure} if any nonzero coherent subsheaf has the
    same dimension.
  \end{enumerate}
\end{defi}

The following lemma describes the relation between pureness of a coherent sheaf
and pureness of its formal model.

\begin{lemm}
\label{formal-rigid pure}
  Let \(F\) be a coherent sheaf on \(X\). The following are equivalent:
   \leavevmode
  \begin{enumerate}
  \item \(F\) is pure;
  \item There is a formal model of \(F\) which is pure;
  \item Any formal model of \(F\) is pure.
  \end{enumerate}
\end{lemm}

  This lemma follows immediately from the fact that the category of coherent
  subsheaves in \(F\) with morphisms being inclusions and the category of
  finitely presented subsheaves in \(\mathcal{F}\) with \(\mathcal{O}\)-flat
  quotient and morphisms being inclusions are equivalent. With one direction
  functor being intersecting with \(\mathcal{F}\) and the functor of taking
  generic fiber in the reverse direction. For the sake of completeness, let us record a proof below.

\begin{proof}
  It suffices to show that given \(\mathcal{F}\), formal model of \(F\), \(F\) is pure if and only if
  \(\mathcal{F}\) is pure. This is because formal model of a coherent sheaf always exists.
  
  Now suppose \(G \subset F\) is a subsheaf with support of smaller dimension than that of \(F\),
  then \(\mathcal{G} \coloneqq \mathcal{F} \cap G\) is a subsheaf of \(\mathcal{F}\) and is a formal model of \(G\).
  Notice that a \(\mathcal{O}\)-flat finitely presented sheaf on \(\mathcal{X}\) has flat support,
  hence the dimension of support of \(\mathcal{G}\) is the same as that of its generic fibre.
  One may check that \(\mathcal{G}\) defined above, as a subsheaf, satisfies the condition that \(\mathcal{F}/\mathcal{G}\) is \(\mathcal{O}\)-flat,
  therefore making \(\mathcal{F}\) not pure.
  
  On the other hand, suppose \(\mathcal{G} \subset \mathcal{F}\) has flat quotient and the dimension of the support of \(\mathcal{G}\)
  is smaller than that of \(\mathcal{F}\). Then by flatness again, we see that the dimension of the support of \(\mathcal{G}^{rig} \eqqcolon G \subset F\)
  is smaller than that of \(F\). Therefore \(F\) is also not pure.
\end{proof}

\begin{defi}
  Let \(\mathcal{F}\) be an \(\mathcal{O}\)-flat finitely presented sheaf on
  \(\mathcal{X}\).
  \begin{enumerate}
  \item \(\mathcal{F}\) is said to be \emph{semistable} if it is pure and for
    any finitely presented subsheaf \(\mathcal{G} \subset \mathcal{F}\) we have
    \(p(\mathcal{G}) \leq p(\mathcal{F})\), where \(p(\mathcal{F})\) is the
    reduced Hilbert polynomial, i.e. Hilbert polynomial divided by its leading
    coefficient, of a finitely presented sheaf \(\mathcal{F}\). Here \(p(\mathcal{G}) \leq p(\mathcal{F})\) 
    means coefficient-wise.
    We denote the
    \((k+1)\)-th coefficient of reduced Hilbert polynomial of \(\mathcal{F}\) by
    \(a_k(\mathcal{F})\).
  \item Let \(f=x^n+b_1x^{n-1}+\cdots+b_n\) be the reduced Hilbert polynomial of
    the maximal destabilizing sheaf of \(\mathcal{F}_0\). We define the
    \emph{maximal destabilizing sheaf of codimension \(k\)} of \(\mathcal{F}_0\)
    as the maximal subsheaf \(B \subset \mathcal{F}_0\) such that \(a_i(B)=b_i\)
    for all \(1 \leq i \leq k\). The existence of the maximal destabilizing
    sheaf of codimension \(k\) is guaranteed by Harder--Narasimhan theory.
  \item \(\mathcal{F}\) is \emph{semistable of codimension \(k\)} if for any
    coherent subsheaf \(\mathcal{G} \subset \mathcal{F}\) with
    \(\mathcal{O}\)-flat quotient, we have \(a_i(\mathcal{G}) \leq
    a_i(\mathcal{F})\) for all \(1 \leq i \leq k\).
  \item The \emph{semistable codimension} of \(\mathcal{F}\) is the biggest
    \(k\) for which \(\mathcal{F}\) is semistable of codimension \(k\).
  \item We make similar definitions for a coherent sheaf \(F\) on \(X\).
  \end{enumerate}
\end{defi}

For a reference of Harder--Narasimhan theory, one may consult~\cite[Section 1.3]{Huybrechts-Lehn}.

\begin{remark}
  One word on the existence of Harder--Narasimhan filtration in the rigid setup
  is necessary. One can check the proof of existence and uniqueness of
  Harder--Narasimhan filtrations in the algebraic setup (c.f.~\cite[Theorem
  1.3.4]{Huybrechts-Lehn}) to see that we only need Noetherianness of coherent
  sheaves and additivity of Hilbert polynomial (with respect to short exact
  sequences) to do that. In the rigid setup both properties still hold.
\end{remark}

\begin{lemm}
\label{formal-rigidequivalence}
  Let \(F\) be a coherent sheaf on \(X\). The following are equivalent:
\leavevmode
  \begin{enumerate}
  \item \(F\) is semistable of codimension \(k\);
  \item There is a formal model of \(F\) which is semistable of codimension
    \(k\);
  \item Any formal model of \(F\) is semistable of codimension \(k\).
  \end{enumerate}
\end{lemm}

\begin{proof}
  The proof is the same as that of the previous lemma.
\end{proof}  

\begin{defi}
  Let \(\mathcal{F}\) be an \(\mathcal{O}\)-flat finitely presented sheaf on
  \(\mathcal{X}\). Consider an exact sequence \(0 \to \mathcal{B}_0 \to
  \mathcal{F}_0 \to \mathcal{G}_0 \to 0\) of coherent sheaves on
  \(\mathcal{X}_0\). We say such a sequence is \emph{liftable modulo \(\pi \in
    \mathfrak{m}\)} if there is an exact sequence of finitely presented sheaves
  \(0 \to \mathcal{B} \to \mathcal{F}\otimes \mathcal{O}/(\pi) \to \mathcal{G}
  \to 0\) whose special fiber is the given sequence above with \(\mathcal{G}\)
  flat over \(\mathcal{O}/(\pi)\).
\end{defi}

\begin{lemm}
  Let \(\mathcal{F}\) be an \(\mathcal{O}\)-flat finitely presented sheaf on
  \(\mathcal{X}\). Consider \(0 \to \mathcal{B}_0 \to \mathcal{F}_0 \to
  \mathcal{G}_0 \to 0\), a short exact sequence of coherent sheaves on
  \(\mathcal{X}_0\).
\label{flip lemma}
  \leavevmode
  \begin{enumerate}
  \item If \(\Hom(\mathcal{B}_0,\mathcal{G}_0)=\{0\}\), there exists a biggest
    (in the sense of having maximal valuation) \(\pi \in \mathfrak{m}\) such
    that the sequence is liftable modulo \(\pi\). Here \(\pi\) could be \(0\),
    by which we mean the sequence can be lifted all the way to \(\mathcal{O}\).
  \item In the following, assume the \(\pi\) above is not \(0\). Let
    \(\mathcal{F}^{(1)}\) be the kernel of the composition \(\mathcal{F} \to
    \mathcal{F}\otimes \mathcal{O}/(\pi) \to \mathcal{G}\). Then
    \(\mathcal{F}^{(1)}\) is finitely presented, and we have a short exact
    sequence \(0 \to \mathcal{G} \to \mathcal{F}^{(1)}\otimes \mathcal{O}/(\pi)
    \to \mathcal{B} \to 0\). Furthermore, \(\mathcal{F}/\mathcal{F}^{(1)}\) is
    \(\pi\)-torsion.
  \item Moreover the reduction of our short exact sequence above \(0 \to
    \mathcal{G}_0 \to \mathcal{F}^{(1)}_0 \to \mathcal{B}_0 \to 0\) does not
    split.
    \end{enumerate}
\end{lemm}

\begin{proof}
  Proof of (1). For any \(\pi \in \mathfrak{m} \backslash \{0\}\), we may consider the Quot scheme
  \(\mathcal{Q}_\pi=\mathrm{Quot}_{\mathcal{F}_\pi/\mathcal{X}_\pi/\mathcal{O}_\pi}\) 
  \footnote{For a discussion about Quot schemes, one may consult~\cite[Section 2.2]{Huybrechts-Lehn}, \cite{Nit}
   and~\cite[\href{https://stacks.math.columbia.edu/tag/09TQ}{Section 09TQ}]{stacks-project}.}.
  They are locally finitely presented schemes over \(\mathrm{Spec}(\mathcal{O}/(\pi))\) satisfying
  base change proper: \(\mathcal{Q}_{\pi} \otimes_{\mathcal{O}/(\pi)} \mathcal{O}/\pi' = \mathcal{Q}_{\pi'}\)
  whenever \(0 < |\pi| < |\pi'| < 1\). Hence they form a locally finitely presented formal scheme
  \(\mathcal{Q}=\mathrm{Quot}_{\mathcal{F}/\mathcal{X}/\mathcal{O}}\). 
  The special fibre \(\mathcal{Q}_0 = \mathcal{Q} \otimes_{\mathcal{O}} k\) of \(\mathcal{Q}\) is
  the Quot scheme \(\mathrm{Quot}_{\mathcal{F}_0/\mathcal{X}_0/\mathcal{O}_0}\). Let us denote the point
  corresponding to \(\mathcal{F}_0 \twoheadrightarrow \mathcal{G}_0\) by \(P \in \mathcal{Q}_0(k)\).
  By~\cite[Proposition 2.2.7]{Huybrechts-Lehn}, we know that our condition implies the tangent space
  \(T_{P}\mathcal{Q}_0\) is zero dimensional. 
  Hence we know the component of \(\mathcal{Q}\) corresponding to \(P\) is of the form 
  \(\mathcal{Q}_p=\Spec(\mathcal{O}/(\pi))\) which means
  exactly that our sequence is liftable modulo \(\pi\) but not modulo an element
  with bigger valuation.

  Proof of (2). We consider \(\pi\mathcal{F}^{(1)} \subset \pi\mathcal{F}
  \subset \mathcal{F}^{(1)}\), whose corresponding quotients give us a short
  exact sequence
  \[
    0 \to \mathcal{G}=\pi\mathcal{F}/\pi\mathcal{F}^{(1)} \to
    \mathcal{F}^{(1)}\otimes
    \mathcal{O}/(\pi)=\mathcal{F}^{(1)}/\pi\mathcal{F}^{(1)} \to
    \mathcal{B}=\mathcal{F}^{(1)}/\pi\mathcal{F} \to 0
  \]
  which is what we want, and it is immediate that
  \(\mathcal{F}/\mathcal{F}^{(1)}\) is \(\pi\)-torsion.

  Proof of (3). Suppose on the contrary that the sequence \(0 \to \mathcal{G}_0
  \to \mathcal{F}^{(1)}_0 \to \mathcal{B}_0 \to 0\) splits, so that we have 
  \(\mathcal{F}^{(1)}_0 \simeq \mathcal{G}_0 \oplus \mathcal{B}_0\) and we may view
  \(\mathcal{B}_0\) as a subsheaf in \(\mathcal{F}^{(1)}_0\). Because the maximal ideal
   \(\mathfrak{m}\) is the colimit of principal sub-ideals \((\pi')\),
  by a limit
  argument we have \(0 \to \mathcal{B}' \to
  \mathcal{F}^{(1)}/\pi'\mathcal{F}^{(1)} \to \mathcal{G}' \to 0\) for some
  \(\pi' \in \mathfrak{m}\) where \(\mathcal{G}'\) is flat over
  \(\mathcal{O}/\pi'\), and we may choose \(\pi'\) so that the composite
  \(\pi\mathcal{F} \hookrightarrow \mathcal{F}^{(1)} \twoheadrightarrow
  \mathcal{G}'\) is surjective. Then we consider \(\tilde{\mathcal{B}}\) which
  is defined by the following exact sequences
    \[
      \xymatrix{ 0 \ar[r] & \tilde{\mathcal{B}} \ar[d] \ar[r] &
        \mathcal{F}^{(1)}/\pi\pi'\mathcal{F} \ar[d] \ar[r] & \mathcal{G}' \ar[d]
        \ar[r] & 0 \\
        0 \ar[r] & \mathcal{B}' \ar[r] & \mathcal{F}^{(1)}/\pi'\mathcal{F}^{(1)}
        \ar[r] & \mathcal{G}' \ar[r] & 0 }
    \]
    We claim that the exact sequence
    \[ 0 \to \tilde{\mathcal{B}} \to \mathcal{F}/\pi\pi'\mathcal{F} \to
      \tilde{\mathcal{G}} \to 0
    \]
    is flat over \(\mathcal{O}/\pi\pi'\), which would contradict the maximality
    of (the valuation of) \(\pi\). Indeed, by considering \(\tilde{\mathcal{B}}
    \subset \mathcal{F}^{(1)}/\pi\pi'\mathcal{F} \subset
    \mathcal{F}/\pi\pi'\mathcal{F}\), we get \(0 \to \mathcal{G}' \to
    \tilde{\mathcal{G}} \to \mathcal{G} \to 0\). To prove
    \(\tilde{\mathcal{G}}\) is flat over \(\mathcal{O}/\pi\pi'\) it suffices to
    show that the map \(\tilde{\mathcal{G}} \to \mathcal{G}\) is the same as
    tensoring \(\mathcal{O}/\pi\), which in turn is equivalent to saying
    \(\tilde{\mathcal{B}} + \pi\mathcal{F}/\pi\pi'\mathcal{F} =
    \mathcal{F}^{(1)}/\pi\pi'\mathcal{F}\), and this equality follows from the
    fact that the composite \(\pi\mathcal{F} \hookrightarrow \mathcal{F}^{(1)}
    \to \mathcal{G}'\) is surjective.
\end{proof}

The key result of this section is the following theorem which generalizes
Langton's result in~\cite{Langton}. This theorem and the proof of it is very much
inspired by~\cite[Section 2.B]{Huybrechts-Lehn}.

\begin{thm}
  Let \(X\) be a rigid space with a formal model \(\mathcal{X}\). Assume
  \(\mathcal{X}_0\) is projective with a fixed ample invertible sheaf \(H\).
\label{semistable reduction}  
    \leavevmode
  \begin{enumerate}
  \item Let \(F\) be a pure coherent sheaf on \(X\). Then there exists a formal
    model \(\mathcal{F}\) on \(\mathcal{X}\) with reduction \(\mathcal{F}_0\)
    pure on \(\mathcal{X}_0\). Conversely, if \(F\) admits a formal model
    \(\mathcal{F}\) with pure reduction, then \(F\) itself is pure.
  \item Let \(F\) be a semistable coherent sheaf on \(X\). Then there exists a
    formal model \(\mathcal{F}\) on \(\mathcal{X}\) with reduction
    \(\mathcal{F}_0\) semistable on \(\mathcal{X}_0\). Conversely, if \(F\)
    admits a formal model \(\mathcal{F}\) with semistable reduction, then \(F\)
    itself is semistable.
  \end{enumerate}
\end{thm}

Before giving the proof, let us say some informal idea of the proof here.
We want to construct a formal model of our coherent sheaf so that its special fibre is as nice as the generic fibre.
The idea is to start with an arbitrary formal model and keep modifying it so that 
at each step the special fibre becomes better with respect to the properties we want it to have. 
This is achieved by using Lemma~\ref{flip lemma} above. 
Lastly we have to show that this process must terminate in finitely many steps. 
The key technique we use to achieve this is provided by the duality theory, 
especially the theory of reflexive hulls (c.f.~\cite[Section 1.1]{Huybrechts-Lehn}).

\begin{proof}
  Proof of (1). We choose an arbitrary formal model \(\mathcal{F}\) on
  \(\mathcal{X}\); by Lemma~\ref{formal-rigid pure} we know \(\mathcal{F}\) is
  pure. Now we want to find a finitely presented subsheaf \(\mathcal{F}' \subset
  \mathcal{F}\) with torsion quotient and special fiber \(\mathcal{F}'_0\) pure
  on \(\mathcal{X}_0\).

  Let \(\mathcal{B}_0 \subset \mathcal{F}_0\) be the maximal coherent subsheaf
  whose support is not of dimension \(\dim(F)\) (from now on we will call it the
  maximal torsion subsheaf of \(\mathcal{F}_0\)) and denote the quotient by
  \(\mathcal{G}_0\). It is easy to see that
  \(\Hom(\mathcal{B}_0,\mathcal{G}_0)=0\), so applying Lemma~\ref{flip lemma} to
  the short exact sequence \(0 \to \mathcal{B}_0 \to \mathcal{F}_0 \to
  \mathcal{G}_0 \to 0\) gives a finitely presented subsheaf \(\mathcal{F}^{(1)}
  \subset \mathcal{F}\) and a short exact sequence \(0 \to \mathcal{G}_0 \to
  \mathcal{F}^{(1)}_0 \to \mathcal{B}_0 \to 0\). After repeatedly applying the
  above procedure, we will obtain a sequence of finitely presented sheaves
  \[\mathcal{F}=\mathcal{F}^{(0)} \supset \mathcal{F}^{(1)} \supset \cdots.
  \]
  Note that all of the \(\mathcal{F}^{(i)}_0\)'s have the same Hilbert
  polynomial. We get short exact sequences
  \[ 0 \to \mathcal{B}^{(i)}_0 \to \mathcal{F}^{(i)}_0 \to \mathcal{G}^{(i)}_0
    \to 0
  \]
  and
  \[
    0 \to \mathcal{G}^{(i)}_0 \to \mathcal{F}^{(i+1)}_0 \to \mathcal{B}^{(i)}_0
    \to 0.
  \]
  We claim that after finitely many steps the maximal torsion subsheaf of
  \(\mathcal{F}^{(i)}_0\), denoted by \(\mathcal{B}^{(i)}_0\), will have either
  dimension or rank smaller than that of \(\mathcal{B}_0\). According to this
  claim, after \(N\) steps, \(\mathcal{B}^{(N)}_0\) will be 0. This implies
  \(\mathcal{F}^{(N)}_0\) contains no maximal torsion subsheaf, hence is pure.

  Suppose otherwise. Then because \(\mathcal{G}^{i}_0 \cap
  \mathcal{B}^{(i+1)}_0=0\), we have an infinite chain of inclusions
  \[\mathcal{G}=\mathcal{G}^{(0)}_0 \hookrightarrow \mathcal{G}^{(1)}_0
    \hookrightarrow \cdots
  \]
  and
  \[\mathcal{B}=\mathcal{B}^{(0)}_0 \hookleftarrow \mathcal{B}^{(1)}_0
    \hookleftarrow \cdots.
  \]
  We are assuming all of the \(\mathcal{B}^{(i)}_0\)'s have the same dimension
  and rank as those of \(\mathcal{B}_0\).
  
  By Lemma~\ref{flip lemma} (3), all the inclusions above are not isomorphisms.
  We notice that \(P_H(\mathcal{G}^{(i+1)}_0)-P_H(\mathcal{G}^{(i)}_0) =
  P_H(\mathcal{B}^{(i)}_0)-P_H(\mathcal{B}^{(i+1)}_0)\) and the dimension of
  \(\mathcal{B}^{(i)}_0\) is smaller than the dimension of
  \(\mathcal{G}^{(i)}_0\). Hence by our assumptions, the
  \(\mathcal{G}^{(i)}_0\)'s only differ in codimension \(\geq 2\), therefore
  they have the same reflexive hull \(\mathcal{G}_0^{DD}\). Viewing
  \(\mathcal{G}^{(i)}_0\) as an infinite increasing chain of subsheaves in
  \(\mathcal{G}_0^{DD}\), we find a contradiction with the Noetherianness of
  \(\mathcal{G}_0^{DD}\).

  The converse part is easy. Suppose \(G \subset F\) is a coherent subsheaf with
  support of lower dimension than \(\dim(F)\). Then \(\mathcal{G} := G \cap
  \mathcal{F}\) gives a contradiction to the assumption that \(\mathcal{F}_0\)
  is pure.

  Proof of (2). By the first part of our theorem and
  Lemma~\ref{formal-rigidequivalence}, we can choose a formal model
  \(\mathcal{F}\) such that it is semistable and its special fiber
  \(\mathcal{F}_0\) is pure on \(\mathcal{X}_0\). Now we want to find a finitely
  presented subsheaf \(\mathcal{F}' \subset \mathcal{F}\) with torsion quotient
  and special fiber \(\mathcal{F}'_0\) semistable on \(\mathcal{X}_0\). Starting
  from \(\mathcal{F}=\mathcal{F}^{(0)}\), we will do induction on the semistable
  codimension of \(\mathcal{F}_0\) which we denoted as \(k(\mathcal{F}_0)\).

  We denote the maximal destabilizing sheaf of codimension
  \(k(\mathcal{F}_0)+1\) in \(\mathcal{F}_0\) by \(\mathcal{B}_0\), and the
  quotient by \(\mathcal{G}_0\). It is easy to see that
  \(\Hom(\mathcal{B}_0,\mathcal{G}_0)=0\). Then applying Lemma~\ref{flip lemma}
  to the short exact sequence \(0 \to \mathcal{B}_0 \to \mathcal{F}_0 \to
  \mathcal{G}_0 \to 0\), we get a finitely presented subsheaf
  \(\mathcal{F}^{(1)} \subset \mathcal{F}\) and a short exact sequence \(0 \to
  \mathcal{G}_0 \to \mathcal{F}^{(1)}_0 \to \mathcal{B}_0 \to 0\). It is easy to
  see that \(k(\mathcal{F}^{(1)}_0) \geq k(\mathcal{F}_0)\). After repeatedly
  applying the above procedure we will obtain a sequence of finitely presented
  sheaves
  \[\mathcal{F}=\mathcal{F}^{(0)} \supset \mathcal{F}^{(1)} \supset \cdots.
  \]
  Moreover, we will get short exact sequences
  \[ 0 \to \mathcal{B}^{(i)}_0 \to \mathcal{F}^{(i)}_0 \to \mathcal{G}^{(i)}_0
    \to 0
  \]
  and
  \[
    0 \to \mathcal{G}^{(i)}_0 \to \mathcal{F}^{(i+1)}_0 \to \mathcal{B}^{(i)}_0
    \to 0.
  \]
  We claim that after repeating the above procedure finitely many times
  \(\mathcal{B}^{(i)}_0\), the maximal destabilizing sheaf of codimension
  \(k(\mathcal{F}_0)+1\), will have either rank or \(a_{k+1}\) smaller than that
  of \(\mathcal{B}_0\). Recall that rank takes values in
  \(\frac{\mathbb{N}}{n!}\) and \(a_{k+1}\) takes values in
  \(\frac{\mathbb{Z}}{n!}\). Hence according to our claim, after \(N\) steps,
  \(a_{k+1}(\mathcal{B}^{(N)}_0)=a_{k+1}(\mathcal{F}^{(N)})\). This implies
  \(\mathcal{F}^{(N)}_0\) is semistable of codimension at least
  \(k(\mathcal{F}_0)+1\). By induction we would be done.

  Now we prove our claim. Suppose all of the \(\mathcal{B}^{(i)}_0\)'s have
  non-decreasing rank and \(a_{k+1}\). In that situation, by considering the
  injection \(\frac{\mathcal{B}^{(i+1)}_0}{\mathcal{G}^{i}_0 \cap
    \mathcal{B}^{(i+1)}_0} \hookrightarrow \mathcal{B}^{(i)}_0\) we see that
  \(\mathcal{G}^{i}_0 \cap \mathcal{B}^{(i+1)}_0=0\). Therefore we may assume
  all of the \(\mathcal{B}^{(i)}_0\)'s have the same rank and \(a_{k+1}\) as
  those of \(\mathcal{B}_0\). Hence we will get an infinite chain of inclusions
  \[\mathcal{G}=\mathcal{G}^{(0)}_0 \hookrightarrow \mathcal{G}^{(1)}_0
    \hookrightarrow \cdots
  \]
  and
  \[
    \mathcal{B}=\mathcal{B}^{(0)}_0 \hookleftarrow \mathcal{B}^{(1)}_0
    \hookleftarrow \cdots.
  \]

  By Lemma~\ref{flip lemma} (3), none of the inclusions above are isomorphisms.
  We note that \(P_H(\mathcal{G}^{(i+1)}_0)-P_H(\mathcal{G}^{(i)}_0) =
  P_H(\mathcal{B}^{(i)}_0)-P_H(\mathcal{B}^{(i+1)}_0)\) and by our assumption
  the \(\mathcal{B}^{(i)}_0\)'s only differ in codimension \(\geq 2\). Hence we
  see that the \(\mathcal{G}^{(i)}_0\)'s also only differ in codimension \(\geq
  2\), therefore they have the same reflexive hull \(\mathcal{G}_0^{DD}\).
  Viewing \(\mathcal{G}^{(i)}_0\) as an infinite increasing chain of subsheaves
  in \(\mathcal{G}_0^{DD}\), we find a contradiction with the Noetherianness of
  \(\mathcal{G}_0^{DD}\).
  
  The converse direction is the same as the proof of the first part.
\end{proof}

\begin{thm}
\label{S-equivalence theorem}
  In the situation above, if \(\mathcal{F}^1\) and \(\mathcal{F}^2\) are two
  formal models with semistable reductions, then their reductions
  \(\mathcal{F}^1_0\) and \(\mathcal{F}^2_0\) are S-equivalent.
\end{thm}

\begin{proof}
  Suppose \(\mathcal{F}^1\) and \(\mathcal{F}^2\) are two formal models with
  semistable reductions. Then by Remark~\ref{S-equivalence} we see that there is
  a decreasing filtration on \(\mathcal{F}^1_0\) and an increasing filtration on
  \(\mathcal{F}^2_0\), along with exact sequences
  \[
    0 \to \Fil^{i+1}(\mathcal{F}^1_0) \to \Fil^i(\mathcal{F}^1_0) \to
    \frac{\mathcal{F}^2_0}{\Fil^i(\mathcal{F}^2_0)} \to
    \frac{\mathcal{F}^2_0}{\Fil^{i+1}(\mathcal{F}^2_0)} \to 0.
  \]
  Because both \(\mathcal{F}^1_0\) and \(\mathcal{F}^2_0\) are semistable with
  the same Hilbert polynomial, by induction on the number of filtrations we see
  that \(\Fil^i(\mathcal{F}^1_0)\), \(\mathcal{F}^2_0/\Fil^i{\mathcal{F}^2_0}\),
  \(\Gr^i(\mathcal{F}^1_0)\) and \(\Gr^i(\mathcal{F}^2_0)\) are all semistable
  with the same Hilbert polynomial. Now by Remark~\ref{S-equivalence} we have
  isomorphisms between \(\Gr^i(\mathcal{F}^1_0)\) and
  \(\Gr^i(\mathcal{F}^2_0)\). Putting the above together, we see that
  \(\mathcal{F}^1_0\) and \(\mathcal{F}^2_0\) are S-equivalent.
\end{proof}

Similarly one can prove the following theorem, which generalizes Langton's
theorem to the case of arbitrary valuations.

\begin{thm}
  Let \(X \to \Spec(R)\) be a flat projective finitely presented scheme over a
  valuation ring \(R\). Let \(F_K\) be a semistable sheaf on \(X_K\). Then there
  exists a coherent sheaf \(F\) on \(X\) with generic fiber \(F_K\) and special
  fiber \(F_0\) semistable on \(X_0\). Moreover any two such coherent sheaves
  have S-equivalent special fibers.
\end{thm}

\begin{proof}
  The proof is almost the same, the only subtlety being that a finitely
  presented algebra over a valuation ring is always coherent
  (c.f.~\cite[Theorem 7.3.3]{Coherent}).
\end{proof}

\begin{remark}
  One can use the above method to prove other semistable reduction type
  theorems. For example, semistable reduction for multi-filtered vector spaces
  or quiver representations. The former case has been worked out by the author
  in~\cite{multi-filtered}. There is a slight subtlety as the category of
  multi-filtered vector spaces is not abelian.
\end{remark}

\begin{remark}
  Every line bundle is automatically stable, so by Theorem~\ref{semistable
    reduction} one can always find a formal model of a given line bundle with a
  semistable reduction. This will be used in the last section to prove the main
  theorem.
\end{remark}

% * Proof
\section{Proof of the Main Theorem}

% ** Auxiliary space
\subsection{The Auxiliary Space \(W\)}
\label{chosen}

Assume \(\mathcal{X}_0\) is a projective variety over \(k\) and fix an ample
invertible sheaf \(H\). Let \(P\) be the Hilbert polynomial of
\(\mathcal{O}_{\mathcal{X}_0}\). Then for every \(\gamma \in \Gamma\) we denote
the moduli stack of finitely presented sheaves with semistable geometric fiber
of Hilbert polynomial \(P\) on \(\mathcal{X}_{\gamma}/\mathcal{O}_{\gamma}\) by
\(\mathcal{M}_{\gamma}\). The existence of such a stack will be justified in the
following, as we do not assume
\(\mathcal{X}_{\gamma}/\Spec(\mathcal{O}_{\gamma})\) to be projective.
Nevertheless we know that \(\mathcal{M}_0\), the algebraic stack of semistable
coherent sheaves of Hilbert polynomial \(P\), is an open substack of
\(|Coh_{\mathcal{X}_0/\Spec(k)}|=|Coh_{\mathcal{X}_{\gamma}/\Spec(\mathcal{O}_{\gamma})}|\).
Therefore we can define \(\mathcal{M}_{\gamma}\) to be the corresponding open
substack in \(Coh_{\mathcal{X}_{\gamma}/\Spec(\mathcal{O}_{\gamma})}\). For
general theory concerning the stack of coherent sheaves, we refer
to~\cite[\href{https://stacks.math.columbia.edu/tag/09DS}{Tag
  09DS}]{stacks-project}.

By the description of the functors we know that \(\mathcal{M}_{\gamma}
\times_{\mathcal{O}_{\gamma}} \mathcal{O}_{\gamma'}=\mathcal{M}_{\gamma'}\) for
any \(\gamma' \leq \gamma\), so the \(\mathcal{M}_{\gamma}\)'s form a projective
system of algebraic stacks. Langer has shown that \(\mathcal{M}_0\) is a
quasi-compact algebraic stack (c.f.~\cite[Theorem 4.4]{Langer}). In particular,
for a chosen pseudo-uniformizer \(\pi \in \mathfrak{m}\) we can find a smooth
surjection \(\mathcal{W}_{\pi} \to \mathcal{M}_{\pi}\) where
\(\mathcal{W}_{\pi}=\Spec(A_1)\) is an affine scheme of finite presentation over
\(\mathcal{O}_{\pi}\). The following is a result of Emerton which is the key of
this subsection.

\begin{prop}[Emerton,~{{\cite[\href{http://stacks.math.columbia.edu/tag/0CKI}{Tag
    0CKI}]{stacks-project}}}]
\label{proposition-affine-smooth-lift-to-first-order}
Let $\mathcal{X} \subset \mathcal{X}'$ be a first order thickening
of algebraic stacks. Let $W$ be an affine scheme and let
$W \to \mathcal{X}$ be a smooth morphism. Then there exists
a cartesian diagram
$$
\xymatrix{
W \ar[d] \ar[r] & W' \ar[d] \\
\mathcal{X} \ar[r] & \mathcal{X}'
}
$$
with $W' \to \mathcal{X}'$ smooth.
\end{prop}

Applying this result yields the following commutative diagrams
\[
  \xymatrix{
    \mathcal{W}_i \ar[d]^{f_i} \ar[r] & \mathcal{W}_{i+1} \ar[d]^{f_{i+1}} \\
    \mathcal{M}_{\pi^i} \ar[r] & \mathcal{M}_{\pi^{i+1}} }
\]
where \(\mathcal{W}_i=\Spec(A_i) \hookrightarrow
\mathcal{W}_{i+1}=\Spec(A_{i+1})\) is the thickening defined by \(\pi^i\). Now
\(\mathcal{A}'=\varprojlim A_i\) is a topologically finitely generated
\(\mathcal{O}\)-algebra, and \(\mathcal{A}=\mathcal{A}'/(\pi \text{-torsions})\)
is a topologically finitely presented \(\mathcal{O}\)-algebra
by~\cite[Proposition 1.1 (c)]{BL1}; denote its formal spectrum
\(\Spf(\mathcal{A})=\mathcal{W}\) and let \(W=\mathcal{W}^{\mathrm{rig}}\) be
its associated rigid analytic space which is an affinoid space. This is the
auxiliary space we want. By the description of functors, we have a system of
\(\mathcal{W}_i\)-flat coherent sheaves \(\mathcal{F}_i\) on
\(\mathcal{X}_{\pi^i} \times_{\mathcal{O}_{\pi^i}} \mathcal{W}_i\). Therefore we
get a \(\mathcal{W}\)-flat finitely presented sheaf \(\mathcal{F}=(\varprojlim
\mathcal{F}_i)/(\pi-torsions)\) (c.f.~\cite[Lemma 1.2 (c)]{BL1}) on
\(\mathcal{X} \times_{\mathcal{O}} \mathcal{W}\), and taking generic fiber gives
us a \(W\)-flat coherent sheaf \(F^{\mathrm{univ}}=\mathcal{F}^{\mathrm{rig}}\)
on \(X \times W\).

We can also define \(R_i=W_i \times_{\mathcal{M}_{\pi^i}} W_i\),
\(\mathcal{R}=(\varinjlim R_i)/(\pi \text{-torsions})\) and
\(R=\mathcal{R}^{\mathrm{rig}}\). Note that we will get an equivalence relation
\(R \rightrightarrows W\).

\begin{question}
  Can one make sense of ``\(W/R\)'' and prove it is the rigid stack of
  semistable coherent sheaves on \(X\) of Hilbert polynomial \(P\)?
\end{question}

% ** Determinant
\subsection{Determinant Construction}

In this subsection, we will explain the determinant construction which
associates a flat family of coherent sheaves on a smooth proper rigid variety to
a flat family of line bundles. This is well known to the experts and is written
down in~\cite{Knudsen-Mumford}. For reader's convenience we will briefly
introduce the construction in the following.

The following lemma is a disguise
of~\cite[\href{http://stacks.math.columbia.edu/tag/068x}{Tag
  068x}]{stacks-project}, and is the starting point of our determinant
construction.

\begin{lemm}
  Let \(X \to S\) be a smooth map of rigid spaces of relative dimension d, and
  let \(F\) be an \(S\)-flat coherent sheaf on \(X\). Then \(F\) is perfect as a
  complex of coherent sheaves on \(X\), i.e., we can find an admissible covering
  of \(X\) such that on each admissible open \(F\) can be resolved by locally
  free coherent sheaves. In fact the length of each resolution is at most \(d\).
\end{lemm}

\begin{proof}
  % First of all, there exists an admissible open covering by affinoids
  % \(\{U_i=Sp(A_i)\}\) (resp. \(\{V_{ij}=Sp(B_{ij})\}\)) of \(S\) (resp. \(X\))
  % such that \(V_{ij}\) are mapped to \(U_i\) and \(F\) on \(V_{ij}\) is given by
  % a finitely generated \(B_{ij}\)-module \(M_{ij}\).

  We may reduce to the situation where \(X=\Sp(B) \to S=\Sp(A)\) is smooth in
  rigid sense and \(F\) is given by a finitely generated \(B\)-module \(M\). Now
  we meet every condition
  in~\cite[\href{http://stacks.math.columbia.edu/tag/068x}{Tag
    068x}]{stacks-project} except for (2), but we have a replacement: for every
  maximal ideal \(\mathfrak{n} \subset A\) the ring \(B \otimes_A
  \kappa(\mathfrak{n})\) has finite global dimension \(\leq d\). Because it
  suffices to check the local tor dimension of \(M\) is \(\leq n\) at maximal
  ideals \(\mathfrak{m} \subset B\) and every maximal ideal in \(\Spec(B)\) is
  mapped to a maximal ideal in \(\Spec(A)\), our replacement condition will make
  the original argument in the stacks project work.
\end{proof}

Let \(X \to S\) be a smooth proper map of rigid spaces of relative dimension
\(d\) where \(S\) is an affinoid and \(F\) an \(S\)-flat coherent sheaf on
\(X\). Then we can find an admissible covering \(\mathcal{U}=\{U_i=\Sp(A_i)\}\)
of \(X\) such that on each \(U_i\) we can find a projective resolution of \(F\)
of length at most \(d\): \[ K_i^\bullet \to F|_{U_i} \to 0.\] Then we define
\(\det(F)|_{U_i}=\bigotimes{\det(K_i^j)}^{\otimes {(-1)}^j}\), where by
\(\det(K)\) of a locally free sheaf \(K\) we mean its top rank self wedge
product. Now on the overlap \(U_{ij}=U_i \cap U_j\), we get two resolutions of
\(F|_{U_{ij}}\). So we get a quasi-isomorphism
\[\Phi_{ij}: K_i^\bullet|_{U_{ij}} \to K_j^\bullet|_{U_{ij}}
\]
which induces a canonical isomorphism
\[\phi_{ij}: (\det(F)|_{U_i})|_{U_{ij}} \to (\det(F)|_{U_j})|_{U_{ij}}.
\]
Moreover \(\phi_{ij}\) only depends on the homotopy class of \(\Phi_{ij}\)
(c.f~\cite[Theorem 1 and Proposition 2]{Knudsen-Mumford}). On triple
intersection \(U_i \cap U_j \cap U_k\) the composition of chosen maps between
resolutions \(\Phi_{ki} \circ \Phi_{jk} \circ \Phi_{ij}\) is homotopic to the
identity, hence the cocycle condition is satisfied automatically. Therefore the
\(\phi_{ij}\)'s give rise to gluing datum of \(\det(F)|_{U_i}\). We just need
the following two easy properties of this determinant construction.

\begin{prop}
\label{determinant-proposition}
  \leavevmode
  \begin{enumerate}
  \item Let \(L\) be a line bundle on \(X\). Then we have a canonical
    isomorphism
    \[L=\det(L).\]
  \item Let \(f: T \to S\) be an arbitrary morphism of rigid spaces. Then we have a
    canonical isomorphism
    \[\det(f^*F)=f^*(\det(F)).\]
  \end{enumerate}
\end{prop}

% ** Proof
\subsection{Proof of the Main Theorem}

Before proving our Main Theorem, let us fix the notations. Let \(K'\) be a
finite extension of \(K\). Denote by \(\{\}'\) the base change of corresponding
objects from \(K\) (resp.\ \(\mathcal{O}\)) to \(K'\) (resp.\
\(\mathcal{O}'=\mathcal{O}_{K'}\)). Recall that in Subsection~\ref{chosen}, 
right after Proposition~\ref{proposition-affine-smooth-lift-to-first-order}, we have introduced \(W\) and \(\mathcal{F}^{\mathrm{univ}}\).

\begin{proof}[Proof of the Main Theorem~\ref{Main Theorem}]
  It suffices to prove that \(\underline{\Pic}^{P}_{X/K}\) is a bounded functor
  by Theorem~\ref{partially proper}. Consider \(F^{\mathrm{univ}}\) on \(X
  \times W\) which is a \(W\)-flat coherent sheaf. The determinant construction
  gives us a \(W\)-flat line bundle on \(X \times W\), hence a map \(\det: W \to
  \underline{\Pic}^{P}_{X/K}\). We claim this map is a surjection on classical
  Tate points. This means that for every \(T \in
  \underline{\Pic}^{P}_{X/K}(K')\) we can find a preimage of \(T\) in
  \(W(\bar{K'})\), where \(K'\) is a finite extension of \(K\) with
  \(\mathcal{O}' \cap K=\mathcal{O}\).

  Indeed, \(T\) corresponds to a line bundle \(L\) on \(X_{K'}\) with Hilbert
  polynomial the same as that of \(\mathcal{O}_{X}\). Then by
  Theorem~\ref{semistable reduction} we see that \(L\) has a formal model
  \(\mathcal{F}'\) on \(\mathcal{X}_{\mathcal{O}'}\) where \(\mathcal{F}'_0\) is
  semistable. Note that \(\mathcal{F}'_0\) automatically has Hilbert polynomial
  the same as that of \(\mathcal{O}_{\mathcal{X}_0}\) by
  Theorem~\ref{well-defined}. \(\mathcal{F}'\) gives rise to the maps \(s_i\)
  from bottom row to middle row in the diagram~\ref{commutative-diagram} below,
  where the subscript \(i\) here means the \(\pi^{i+1}\)-level of corresponding
  objects with \(\pi\) chosen as in Section~\ref{chosen}. As the map
  \(\mathcal{W}'_0 \to \mathcal{M}'_0\) is surjective and of finite
  presentation, after a further (unramified) finite extension of \(K'\) (for
  simplicity we will still call it \(K'\) below) we can lift the map \(s_0\) to
  \(\sigma_0: \Spec(\mathcal{O}'/(\pi)) \to \mathcal{W}'_0\). By smoothness of
  \(f_i\) we can thus lift the maps \(s_i\) to \(\sigma_i:
  \Spec(\mathcal{O}'/(\pi'^{i+1})) \to \mathcal{W}'_i\).
  \begin{equation}
   \label{commutative-diagram} 
   \xymatrix{
     \mathcal{W}'_0 \ar[d]^{f_0} \ar@{^{(}->}[r] & \mathcal{W}'_1 \ar[d]^{f_1} \ar@{^{(}->}[r] & \cdots \\
     \mathcal{M}'_0 \ar[d] \ar@{^{(}->}[r] & \mathcal{M}'_1 \ar[d] \ar@{^{(}->}[r] & \cdots \\
     \Spec(\mathcal{O}'/(\pi)) \ar@/^/[u]^{s_0} \ar@{-->}@/^2pc/[uu]^{\sigma_0} \ar@{^{(}->}[r] &
     \Spec(\mathcal{O}'/(\pi^2)) \ar@/^/[u]^{s_1} \ar@{-->}@/^2pc/[uu]^{\sigma_1} \ar@{^{(}->}[r] & \cdots
   }
  \end{equation}
  Therefore we get a map \(\sigma: \mathcal{O}' \to \mathcal{W}\), taking
  generic fiber gives us a \(K'\)-point \(Q \in W(K')\). By
  Proposition~\ref{determinant-proposition} we see that \(\det(Q)=T\). This
  completes our proof.
\end{proof}

\begin{remark}
  Our method actually proves \(W \to \underline{\Pic}^{P}_{X/K}\) is a
  surjection on analytic (i.e.~rank \(1\)) points.
\end{remark}

\begin{proof}[Proof of Theorem~\ref{proper picard}]
  Since \(\Pic^0_{X/K}\) is a connected component of
  \(\Pic^P_{X/K}\), it suffices to show the latter space is proper.
  This follows from our Main Theorem~\ref{Main Theorem} and
  Proposition~\ref{justify}.
\end{proof}

\section*{Acknowledgements}

The author wants to express his deep gratitude to his advisor Johan de Jong for
introducing him to the study of rigid geometry, suggesting a thesis problem
which motivates this paper, and providing detailed feedback on an earlier draft.
The author also wants to thank David Hansen for many helpful discussions, in
addition to pointing out the reference~\cite[Theorem 2.3.3]{K-L} and constantly
urging the author to write this paper. He is thankful to Evan Warner for
clarifying some details of the representability of the Picard functor. He thanks
his friends Pak-Hin Lee and Dingxin Zhang as well, for LaTeX and grammar
checking.

\end{document}